\documentclass[10pt]{amsart}
\usepackage[cp1251]{inputenc}
\usepackage[T2A]{fontenc}
\usepackage[ukrainian,russian,english]{babel}
\usepackage{mathrsfs}
\usepackage{latexsym}
\usepackage{cite}
\usepackage{amsmath}
\usepackage{amsfonts}
\usepackage{amssymb}
\usepackage{amsthm}
\usepackage{euscript}
\usepackage[all]{xy}
\usepackage{delarray}
\usepackage{amscd}
\usepackage{euscript}
\usepackage[active]{srcltx}
\usepackage{citehack}
\theoremstyle{plain}
\newtheorem{theorem}{Theorem}
\newtheorem{theoremMain}{Theorem}

\newtheorem{lemma}[theorem]{Lemma}
\newtheorem{proposition}[theorem]{Proposition}

\newtheorem{corollaryThMn}{Corollary}[theoremMain]

\newtheorem{corollaryTh}{Corollary}[theorem]
\theoremstyle{definition}

\theoremstyle{remark}

\newtheorem{remarkTh}{Remark}[theorem]

\theoremstyle{plain}
\newtheorem*{theorem*}{Theorem}
\newtheorem*{lemma*}{Lemma}
\newtheorem*{proposition*}{Proposition}
\newtheorem*{statement*}{Statement}
\newtheorem*{corollary*}{Corollary}
\theoremstyle{definition}
\newtheorem*{definition*}{Definition}
\theoremstyle{remark}
\newtheorem*{notation*}{Notation}
\newtheorem*{remark*}{Remark}
\newtheorem*{example*}{Example}
\renewcommand{\thetheoremMain}%
{\Alph{theoremMain}}

\renewcommand{\thetheoremMainBis}%
{\Alph{theoremMainBisRus}$^{\mathrm{bis}}$}

\begin{document}
\title[Schr\"{o}dinger operators with complex singular potentials]{Schr\"{o}dinger operators with complex singular
potentials}

%----------------------------------- Author 1 ------------------------------
\author[V. Mikhailets]{Vladimir Mikhailets}

\address{Institute of Mathematics \\
         National Academy of Science of Ukraine \\
         3 Tereshchenkiv\-s'ka Str. \\
         01601 Kyiv-4 \\
         Ukraine}

\email{mikhailets@imath.kiev.ua}

%----------------------------------- Author 2 ------------------------------
\author[V. Molyboga]{Volodymyr Molyboga}

\address{Institute of Mathematics \\
         National Academy of Science of Ukraine \\
         3 Tereshchenkiv\-s'ka Str. \\
         01601 Kyiv-4 \\
         Ukraine}

\email{molyboga@imath.kiev.ua}

%%%%%%%%%%%%%%%%%%%%%%%%%%%%%%%%%%%%%%%%%%%%%%%%%%%%%%%%%%%%%%%%%%%%%%%%%%%%%%%%%%%%%%%%%%%%%%%%%%%%%%%%%%%%%%%%%%%%%%%%%%%%
%%%%%%%%%%%%%%%%%%%%%%%%%%%%%%%%%%%%%%%%%%%%%%%%%%%%%%%%%%%%%%%%%%%%%%%%%%%%%%%%%%%%%%%%%%%%%%%%%%%%%%%%%%%%%%%%%%%%%%%%%%%%



\dedicatory{To Myroslav Lvovych Gorbachuk on the occasion of his 75th birthday.}

%\thanks{The authors were partially supported by the grant no. 01-01-12 of National Academy of Science of Ukraine (under the
%joint Ukrainian--Russian project of NAS of Ukraine and Russian Foundation of Basic Research)}

%-------------------------------------------------------------------------------
% 34-XX | ORDINARY DIFFERENTIAL EQUATIONS
% 34Lxx | Ordinary differential operators
% 34L40 | Particular operators (Dirac, one-dimensional Schr\"{o}dinger, etc.)
%-------------------------------------------------------------------------------
% 47-XX | OPERATOR THEORY
% 47Axx | General theory of linear operators
% 47A10 | Spectrum, resolvent
% 47A75 | Eigenvalue problems [See also 47J10, 49R05]
%-------------------------------------------------------------------------------

%%%%%%%%%%%%%%%%%%%%%%%%%%%%%%%%%%%%%%%%%%%%%%%%%%%%%%%%%%%%%%%%%%%%%%%%%%%%%%%%%%%%%%%%%%%%%%%%%%%%%%%%%%%%%%%%%%%%%%%%%%%%
%%%%%%%%%%%%%%%%%%%%%%%%%%%%%%%%%%%%%%%%%%%%%%%%%%%%%%%%%%%%%%%%%%%%%%%%%%%%%%%%%%%%%%%%%%%%%%%%%%%%%%%%%%%%%%%%%%%%%%%%%%%%

\begin{abstract}
We study one-dimensional Schr\"{o}dinger operators $\mathrm{S}(q)$ on the space $L^{2}(\mathbb{R})$
with potentials $q$ being complex-valued generalized functions from the negative space $H_{unif}^{-1}(\mathbb{R})$.
Particularly the class $H_{unif}^{-1}(\mathbb{R})$ contains periodic and almost periodic
$H_{loc}^{-1}(\mathbb{R})$-functions. We establish an equivalence of the various definitions of the operators
$\mathrm{S}(q)$, investigate their approximation by operators with smooth potentials from the space
$L_{unif}^{1}(\mathbb{R})$ and prove that the spectrum of each operator $\mathrm{S}(q)$ lies within a certain parabola.
\end{abstract}

\maketitle

%%%%%%%%%%%%%%%%%%%%%%%%%%%%%%%%%%%%%%%%%%%%%%%%%%%%%%%%%%%%%%%%%%%%%%%%%%%%%%%%%%%%%%%%%%%%%%%%%%%%%%%%%%%%%%%%%%%%%%%%%%%%
%%%%%%%%%%%%%%%%%%%%%%%%%%%%%%%%%%%%%%%%%%%%%%%% !!! Text of the paper !!! %%%%%%%%%%%%%%%%%%%%%%%%%%%%%%%%%%%%%%%%%%%%%%%%%

\section{Introduction and main results}\label{sec:Int}
In the complex Hilbert space $L^{2}(\mathbb{R})$ we consider a Schr\"{o}dinger operator
\[\mathrm{S}(q)=-\dfrac{d^2}{dx^2}+q(x)\]
with potential $q$ that is a complex-valued distribution from the space
$H_{unif}^{-1}(\mathbb{R}) \subset H_{loc}^{-1}(\mathbb{R})$.
Recall that $H_{loc}^{-1}(\mathbb{R})$ is a dual to the space $H_{comp}^1(\mathbb{R})$ of functions
in $H^1(\mathbb{R})$ with compact support
and that every $q\in H_{loc}^{-1}(\mathbb{R})$ can be represented as $Q'$ for $Q\in L_{loc}^2(\mathbb{R})$.
Then the operator $\mathrm{S}(q)$ can be rigorously defined e. g. by so-called regularization method
that was used in \cite{AtEvZt1988} in the particular case $q(x) = 1/x$
and then developed for generic distributional potential functions in $H_{loc}^{-1}(\mathbb{R})$ in \cite{SvSh1999, SvSh2003};
see also recent extensions to more general differential expressions in \cite{GrMi2010, GrMi2011}.
Namely, the regularization method suggests to define $\mathrm{S}(q)$ via
\begin{equation}\label{eq_10}
\mathrm{S}(q) y = l[y] = - (y' - Qy)' - Qy'
\end{equation}
on the natural maximal domain
\begin{equation}\label{eq_11}
\mathrm{Dom} (\mathrm{S}(q)) =
\left\lbrace y \in L^2(\mathbb{R})\;\left|\;y, y' - Qy\in \mathrm{AC}_{loc}(\mathbb{R}),\;
	l[y] \in L^2(\mathbb{R})\right.\right\rbrace,
\end{equation}
here $\mathrm{AC}_{loc}(\mathbb{R})$ is the space of functions that are locally absolutely continuous.
It is easy to see that $\mathrm{S}(q)y = - y'' + qy$ in the sense of distributions and the above definition
does not depend on the particular choice of the primitive $Q \in L_{loc}^2(\mathbb{R})$.

One can also introduce the minimal operator $\mathrm{S}_0(q)$, which is the closure of the restriction
$\mathrm{S}_{00}(q)$ of $\mathrm{S}(q)$ onto the set of functions with compact support, i. e. onto
\[ \mathrm{Dom}(\mathrm{S}_{00}(q)) :=\left\lbrace y\in L_{comp}^2(\mathbb{R})\;
  \left|\;y, y' - Qy\in \mathrm{AC}_{loc}(\mathbb{R}),\;
  	l[y] \in L^2(\mathbb{R})\right.\right\rbrace. \]

In the case when potential function $q$ is real-valued the operator $\mathrm{S}_{00}(q)$
(and hence $\mathrm{S}_0(q)$) is symmetric;
moreover, in a standard manner \cite{MiMl6} one can prove that $\mathrm{S}(q)$ is adjoint of $\mathrm{S}_0(q)$.
An important question preceding any further analysis of the operator $\mathrm{S}(q)$ is whether it is self-adjoint,
i. e. $\mathrm{S}(q)=\mathrm{S}_0(q)$.
The case when the potential belongs to the space $H_{unif}^{-1}(\mathbb{R})$ was investigated in \cite{HrMk2001}.
We recall \cite{HrMk2001} that any $q \in H_{unif}^{-1}(\mathbb{R})$ can be represented (not uniquely) in the form
\begin{equation}\label{eq_111}
q = Q' + \tau,
\end{equation}
where derivative is understood in the sense of distributions and $Q$ and $\tau$ belong to Stepanov spaces 
$L_{unif}^2(\mathbb{R})$ and $L_{unif}^1(\mathbb{R})$ respectively, i. e.
\begin{align*}
\|Q\|_{L_{unif}^2(\mathbb{R})}^2 &:= \sup\limits_{t\in\mathbb{R}}\int\limits_{t}^{t+1}|Q(s)|^2ds < \infty,\\
\|\tau\|_{L_{unif}^1(\mathbb{R})} &:= \sup\limits_{t\in\mathbb{R}}\int\limits_{t}^{t+1}|\tau(s)|ds < \infty.
\end{align*}
Given such a representation, the operator $\mathrm{S}$ is defined as
\begin{equation}\label{eq_112}
\mathrm{S}(q)y = - (y' - Qy) ' - Qy' +\tau y
\end{equation}
on the domain \eqref{eq_11}.
This definition also does not depend on the particular choice of $Q$ and $\tau$ above.
Theorem~3.5 of the paper \cite{HrMk2001} claims that for real-valued $q \in H_{unif}^{-1}(\mathbb{R})$
the operator $\mathrm{S}(q)$ as defined by \eqref{eq_112} and \eqref{eq_11} is self-adjoint
and coincides with the operator $\mathrm{S}_{fs}(q)$ constructed by the form-sum method.
However the proof given in \cite{HrMk2001} is incomplete.

The fact that $\mathrm{S}(q)$ is indeed self-adjoint is rigorously justified in the paper \cite{MiMl6} for the particular
case where $q \in H_{unif}^{-1}(\mathbb{R})$ is periodic.
The authors prove therein that $\mathrm{S}_0(q)$, $\mathrm{S}(q)$, $\mathrm{S}_{fs}(q)$
and the Friedrichs extension $\mathrm{S}_{F}(q)$ of $\mathrm{S}_0(q)$ all coincide.
However the arguments heavily use periodicity of $q$ and can not be applied to generic real-valued
$q \in H_{unif}^{-1}(\mathbb{R})$.
This gap in the proof of Theorem 3.5 of \cite{HrMk2001} is filled in by the authors in their recent paper \cite{HrMk2012},
see also \cite{KsMl2010}.

This paper deals with the case when the potential $q \in H_{unif}^{-1}(\mathbb{R})$ is complex-valued.
One can easily see that in this case all operators $\mathrm{S}_0(q)$, $\mathrm{S}(q)$, $\mathrm{S}_{fs}(q)$ and
$\mathrm{S}_{F}(q)$ are well-defined and are related by
\begin{equation*}\label{eq_21}
\mathrm{S}_{0}(q)\subset\mathrm{S}_{F}(q)=\mathrm{S}_{fs}(q)\subset\mathrm{S}(q),\quad
\mathrm{Dom}(\mathrm{S}_{F}(q))\subset H^{1}(\mathbb{R}),\;
\mathrm{Dom}(\mathrm{S}(q))\subset H_{loc}^{1}(\mathbb{R})\cap L^2(\mathbb{R}).
\end{equation*}
The main purpose of this paper is to prove that these operators coincide and to investigate their approximation and spectral
properties. Let us state the main results.

\begin{theoremMain}\label{thMn_A}
For every function $q \in H_{unif}^{-1}(\mathbb{R})$ operators $\mathrm{S}_{0}(q)$, $\mathrm{S}(q)$,  $\mathrm{S}_{fs}(q)$
and $\mathrm{S}_{F}(q)$ are $m$-sectorial and coincide.
\end{theoremMain}

Theorem~\ref{thMn_A} allows one to link the known results for the Schr\"{o}dinger operators
in the space $L^2(\mathbb{R})$ which are defined in different ways, see e. g. \cite{Bra1985, Hrz1989, Tip1990, Krt2003,
DjMtOTAA2010}.

In the paper \cite{MiMl6} the authors proved that for every real-valued 1-periodic function
$q \in H_{loc}^{-1}(\mathbb{R})$ a sequence of \textit{smooth} 1-periodic functions $q_n$ exists
such that the sequence of operators $\mathrm{S}(q_n)$ converges to the operator $\mathrm{S}(q)$
in the sense of norm resolvent convergence. It is sufficient to establish
\[ \left\lVert q - q_n \right\rVert_{H^{-1}(0,1)}\rightarrow 0,\quad n\rightarrow\infty. \]
The following theorem generalizes this result in two directions. The potential $q$ may be complex-valued and non-periodic.

\begin{theoremMain}\label{thMn_B}
Let $q$, $q_{n}$, $n\geq 1$, belong to the space $H_{unif}^{-1}(\mathbb{R})$.
Then the sequence of operators $\mathrm{S}(q_{n})$, $n\geq 1$, converges to the operator $\mathrm{S}(q)$ 
in the sense of norm resolvent convergence, $\mathrm{R}(\lambda, \mathrm{S}):=(\mathrm{S}-\lambda\mathrm{Id})^{-1}$:
\begin{equation}\label{eq_20}
 \left\lVert \mathrm{R}(\lambda, \mathrm{S}(q))-\mathrm{R}(\lambda, \mathrm{S}(q_{n}))\right\rVert\rightarrow 0,\quad
 n\rightarrow\infty,\qquad \lambda \in \mathrm{Resolv}(\mathrm{S}(q)) \neq \emptyset,
\end{equation}
if
\begin{equation}\label{eq_22}
 q_{n}\overset{H_{unif}^{-1}(\mathbb{R})}{\longrightarrow} q,\quad n\rightarrow\infty
\end{equation}
or, equivalently,
\begin{equation}\label{eq_24}
 Q_{n}\overset{L_{unif}^{2}(\mathbb{R})}{\longrightarrow} Q,\quad
 \tau_{n}\overset{L_{unif}^{1}(\mathbb{R})}{\longrightarrow} \tau,\qquad n\rightarrow \infty.
\end{equation}
\end{theoremMain}

Since the set $C^{\infty}(\mathbb{R})\cap L_{unif}^{1}(\mathbb{R})$ is dense in the space
$H_{unif}^{-1}(\mathbb{R})$ (see Section~\ref{ssec:PrfB} below), then the following corollary holds.
\begin{corollaryThMn}\label{cr_thMnB_B1}
For every function $q \in H_{unif}^{-1}(\mathbb{R})$  there is a sequence of functions $q_n \in C^{\infty}(\mathbb{R})\cap
L_{unif}^{1}(\mathbb{R})$ such that the limit relation \eqref{eq_20} is true. If the function $q$ is real-valued, then the
functions $q_n$ can be chosen to be real-valued as well.
\end{corollaryThMn}

In particular, if $Q$ and $\tau$ are almost periodic Stepanov functions
then $Q_n$ and $\tau_n$ can be chosen to be trigonometrical polynomials \cite[Theorem~I.5.7.2]{Lvt1953}.
If $Q$ and $\tau$ are are bounded and uniformly continuous on the whole real axis $\mathbb{R}$,
then $Q_{n}$ and $\tau_{n}$ can be chosen to be entire analytic functions
\cite[Theorem~I.1.10.1, Remark]{Lvt1953}.

The following theorem allows one to describe the localization of the spectrum of the operators
$\mathrm{S}(q)$.
\begin{theoremMain}\label{thMn_C}
%Operators $\mathrm{S}(q)$ are $m$-sectorial.
The numerical ranges of operators
$\mathrm{S}(q)$ (and therefore their spectra) lie within the parabola:
\begin{align}
 &\left\lvert \mathrm{Im}\,\lambda\right\rvert  \leq 5K\left(\mathrm{Re}\,\lambda+4(2K+1)^{4}\right)^{3/4}, \label{eq_26} \\
 & K =2\left(\lVert Q\rVert_{L_{unif}^{2}(\mathbb{R})} \notag
 +\lVert\tau\rVert_{L_{unif}^{1}(\mathbb{R})}\right).
\end{align}

If the potential $q$ is real-valued, then the self-adjoint operator $\mathrm{S}(q)$
is bounded below by a number
\[ m(K) = \begin{cases}
  -4K,\quad & \text{if}\quad K \in [0,1/2), \\
  -32K^{4},\quad & \text{if}\quad K\geq 1/2.
 \end{cases}\]
\end{theoremMain}

Note that if a complex-valued potential $q \in H_{unif}^{-1}(\mathbb{R})$ is a periodic generalized function, then the
spectrum of the operator $\mathrm{S}(q)$ lies within a quadratic parabola \cite[Theorem~6]{MiMlRNANU2006}.
A similar result holds for certain complex-valued measures, see \cite{Tip1990} and Section~\ref{ssec:PrfC}, 
formula~\eqref{eq_Tip}.

Similar problems are considered in the papers 
\cite{BrNz2013, EcGsNcTs2012, Glv2012, MiMl5, NmSh1999, Schm2012}. 

%%%%%%%%%%%%%%%%%%%%%%%%%%%%%%%%%%%%%%%%%%%%%%%%%%%%%%%%%%%%%%%%%%%%%%%%%%%%%%%%%%%%%%%%%%%%%%%%%%%%%%%%%%%%%%%%%%%%%%%%%%%%
%%%%%%%%%%%%%%%%%%%%%%%%%%%%%%%%%%%%%%%%%%%%%%%%%%%%%%%%%%%%%%%%%%%%%%%%%%%%%%%%%%%%%%%%%%%%%%%%%%%%%%%%%%%%%%%%%%%%%%%%%%%%

\section{Preliminaries}\label{sec:Prl}
This section contains several statements that are used in the proof of Theorem~\ref{thMn_A}.

We begin with introduction the dual operators $\mathrm{S}_{00}^{+}(q)$ и $\mathrm{S}^{+}(q)$. 

The formally adjoint quasi-differential expression $\mathrm{l}^{+}$ for $\mathrm{l}$ is defined by 
\cite{Ztt1975}:
\begin{align*}
 v^{\{0\}} & := v,\quad v^{\{1\}}:=v'-\overline{Q}v,\quad v^{\{2\}}:=(v^{\{1\}})'+\overline{Q}
 v^{\{1\}}+(\overline{Q^{2}}-\overline{\tau})v, \\
 \mathrm{l}^{+}[v] & :=-v^{\{2\}},\qquad \mathrm{Dom}(\mathrm{l}^{+}):=\left\lbrace
 v:\mathbb{R}\rightarrow\mathbb{C}\;\left|\; v,v^{\{1\}}\in \mathrm{AC}_{loc}(\mathbb{R})\right. \right\rbrace.
\end{align*}
By $\overline{\mbox{}\cdot\mbox{}}$ we denote a complex conjugation.

Then
\begin{align*}
 \mathrm{S}^{+}v\equiv \mathrm{S}^{+}(q)v & :=\mathrm{l}^{+}[v],\hspace{5pt}
   \mathrm{Dom}(\mathrm{S}^{+}):=\left\{v\in L^{2}(\mathbb{R})\,\left|\,v,\,v^{\{1\}}\in
 \mathrm{AC}_{loc}(\mathbb{R}),\mathrm{l}^{+}[v]\in L^{2}(\mathbb{R})\right.\right\}, \\
  \mathrm{S}_{00}^{+}v\equiv\mathrm{S}_{00}^{+}(q)v & :=\mathrm{l}^{+}[v], \hspace{5pt}
   \mathrm{Dom}(\mathrm{S}_{00}^{+}) :=\left\lbrace v\in \mathrm{Dom}(\mathrm{S}^{+})\;
  \left|\;\mathrm{supp}\,v\Subset\mathbb{R}\right.\right\rbrace.
\end{align*}

One can easily see that if $\mathrm{Im}\,q\equiv 0$ then operators 
$\mathrm{S}_{00}(q)$ и $\mathrm{S}_{00}^{+}(q)$, $\mathrm{S}(q)$ and $\mathrm{S}^{+}(q)$ coincide.

\begin{lemma}[Theorem~1, Corollary~1 \cite{Ztt1975}]\label{lm_Prf10}
For arbitrary functions $u\in \mathrm{Dom}(\mathrm{S})$, $v\in \mathrm{Dom}(\mathrm{S}^{+})$ 
and finite interval $[a,b]$ the following equality holds: 
\begin{equation}\label{eq_Prf10}
  \int_{a}^{b}l[u]\overline{v}d\,x-\int_{a}^{b}u\overline{l^{+}[v]}d\,x=[u,v]_{a}^{b},
\end{equation}
where
  \begin{align*}
  [u,v](t) & :=u(t)\overline{v^{\{1\}}(t)}-u^{[1]}(t)\overline{v(t)}, \\
  [u,v]_{a}^{b} & :=[u,v](b)-[u,v](a).
\end{align*}
\end{lemma}
\begin{lemma}\label{lm_Prf12}
For arbitrary functions $u\in \mathrm{Dom}(\mathrm{S})$ and $v\in \mathrm{Dom}(\mathrm{S}^{+})$ 
the following limits exist and are finite:
\begin{equation*}\label{eq_Prf12}
  [u,v](-\infty):=\lim_{t\rightarrow-\infty}[u,v](t),\qquad [u,v](\infty):=\lim_{t\rightarrow\infty}[u,v](t).
\end{equation*}
\end{lemma}
\begin{proof}
Let us fix the number $b$ in the equality \eqref{eq_Prf10} and then pass to the limit as $a\rightarrow-\infty$.
Whereas due to the assumptions of the lemma $u,v,l[u],l^{+}[v]\in L^{2}(\mathbb{R})$, 
the limit $[u,v](-\infty)$ exists and is finite.
Similarly one can prove that the limit $[u,v](\infty)$ exists and is finite.

The Lemma is proved.
\end{proof}
\begin{lemma}[Generalized Lagrange identity]\label{lm_Prf14}
For all functions $u\in \mathrm{Dom}(\mathrm{S})$, $v\in \mathrm{Dom}(\mathrm{S}^{+})$ the equality 
\begin{align}
\int_{-\infty}^{\infty}l[u]\overline{v}d\,x-\int_{-\infty}^{\infty}u\overline{l^{+}[v]}d\,x &
=[u,v]_{-\infty}^{\infty},\label{eq_Prf14} \\
[u,v]_{-\infty}^{\infty} & :=[u,v](\infty)-[u,v](-\infty). \notag
\end{align}
holds. 
\end{lemma}
\begin{proof}
The identity \eqref{eq_Prf14} is true due to Lemma~\ref{lm_Prf10} and Lemma~\ref{lm_Prf12}.
\end{proof}

In the following proposition we describe the properties of minimal and maximal operators and their adjoints. 

\begin{proposition}\label{pr_Prf10}
For the operators $\mathrm{S}$, $\mathrm{S}_{00}$ and $\mathrm{S}^{+}$, $\mathrm{S}_{00}^{+}$ the following statements 
are fulfilled.
\begin{itemize}
  \item [$1^{0}$.] 
  Operators $\mathrm{S}_{00}$ abd $\mathrm{S}_{00}^{+}$ are densely defined in the Hilbert space $L^{2}(\mathbb{R})$.
  \item [$2^{0}$.] 
  The following relations hold:
\begin{equation*}
  \left(\mathrm{S}_{00}\right)^{\ast}=\mathrm{S}^{+},\qquad \left(\mathrm{S}_{00}^{+}\right)^{\ast}=\mathrm{S}.
\end{equation*}
  \item [$3^{0}$.] 
  Operators $\mathrm{S}$, $\mathrm{S}^{+}$ are closed and operators $\mathrm{S}_{00}$, $\mathrm{S}_{00}^{+}$
are closable,
\begin{equation*}
  \mathrm{S}_{0}:=\left(\mathrm{S}_{00}\right)^{\sim},\quad
\mathrm{S}_{0}^{+}:=\left(\mathrm{S}_{00}^{+}\right)^{\sim}.
\end{equation*}
  \item [$4^{0}$.] Domains of operators $\mathrm{S}_{0}$, $\mathrm{S}_{0}^{+}$ may be described in the following way:
\begin{align*}
  \mathrm{Dom}(\mathrm{S}_{0}) & =\left\{u\in \mathrm{Dom}(\mathrm{S}) \left|\,[u,v]_{-\infty}^{\infty}=0\quad \forall
v\in \mathrm{Dom}(\mathrm{S}^{+})\right.\right\}, \\
  \mathrm{Dom}(\mathrm{S}_{0}^{+}) & =\left\{v\in \mathrm{Dom}(\mathrm{S}^{+}) \left|\,[u,v]_{-\infty}^{\infty}=0\quad
\forall u\in  \mathrm{Dom}(\mathrm{S})\right.\right\}.
\end{align*}
  \item [$5^{0}$.] Domains of operators $\mathrm{S}$, $\mathrm{S}_{0}$, $\mathrm{S}_{00}$ and $\mathrm{S}^{+}$,
$\mathrm{S}_{0}^{+}$, $\mathrm{S}_{00}^{+}$
satisfy the following relations:
\begin{align*}
  u\in \mathrm{Dom}(\mathrm{S}) & \Leftrightarrow \overline{u}\in \mathrm{Dom}(\mathrm{S}^{+}), \\
  u\in \mathrm{Dom}(\mathrm{S}_{0}) & \Leftrightarrow \overline{u}\in\mathrm{Dom}(\mathrm{S}_{0}^{+}), \\
  u\in \mathrm{Dom}(\mathrm{S}_{00}) & \Leftrightarrow \overline{u}\in \mathrm{Dom}(\mathrm{S}_{00}^{+}). %\label{eq_16.3}
\end{align*}
\end{itemize}
\end{proposition}

The proof of properties~$1^{0}-4^{0}$ in Proposition~\ref{pr_Prf10} is similar to the proof of similar statements 
for symmetric operators on semi-axis~\cite{Ztt1975}, see also \cite{Nai1969}. 
The property~$5^{0}$ is proved by direct calculation.

We use the following estimates obtained in \cite[Lemma~3.2]{HrMk2001} to prove the main theorems. 
\begin{lemma}\label{lm_Prf16}
Let the functions $Q\in L_{unif}^{2}(\mathbb{R})$, $\tau\in L_{unif}^{1}(\mathbb{R})$ and $u\in H^{1}(\mathbb{R})$. 
Then $\forall\varepsilon\in (0,1]$ and $\forall\eta\in (0,1]$ the estimates hold:
\begin{align*}
 \left|(Q,\overline{u}'u)_{L^{2}(\mathbb{R})}\right| & \leq \lVert
Q\rVert_{L_{unif}^{2}(\mathbb{R})}\left(\varepsilon\lVert u'\rVert_{L^{2}(\mathbb{R})}^{2}+4\varepsilon^{-3}\lVert
u\rVert_{L^{2}(\mathbb{R})}^{2}\right), \\
 \left|(\tau,|u|^{2})_{L^{2}(\mathbb{R})}\right| & \leq \lVert
\tau\rVert_{L_{unif}^{1}(\mathbb{R})}\left(\eta\lVert u'\rVert_{L^{2}(\mathbb{R})}^{2}+8\eta^{-1}\lVert
u\rVert_{L^{2}(\mathbb{R})}^{2}\right).
\end{align*}
\end{lemma}

%%%%%%%%%%%%%%%%%%%%%%%%%%%%%%%%%%%%%%%%%%%%%%%%%%%%%%%%%%%%%%%%%%%%%%%%%%%%%%%%%%%%%%%%%%%%%%%%%%%%%%%%%%%%%%%%%%%%%%%%%%%%

\section{Proofs}\label{sec:Prf}

%\subsection{Предварительные результаты}\label{ssec:Prl}

\subsection{Proof of Theorem~\ref{thMn_A}}\label{ssec:PrfA}
Consider the sesquilinear forms generated by preminimal operators $\mathrm{S}_{00}(q)$:
\begin{align*}\label{eq_Prf18}
 \dot{t}_{\mathrm{S}_{00}}[u,v] & :=(\mathrm{S}_{00}(q)u,v)_{L^{2}(\mathbb{R})}=
 (u',v')_{L^{2}(\mathbb{R})}-(Q,\overline{u}'v+\overline{u}v')_{L^{2}(\mathbb{R})}+
 (\tau,\overline{u}v)_{L^{2}(\mathbb{R})}, \\
 \mathrm{Dom}(\dot{t}_{\mathrm{S}_{00}}) & :=\mathrm{Dom}(\mathrm{S}_{00}(q)).
\end{align*}
To them correspond the quadratic forms 
\begin{equation*}
 \dot{t}_{\mathrm{S}_{00}}[u]=(u',u')_{L^{2}(\mathbb{R})}-(Q,\overline{u}'u+\overline{u}u')_{L^{2}(\mathbb{R})}+
 (\tau,|u|^{2})_{L^{2}(\mathbb{R})}.
\end{equation*}
We introduce the notation:
\begin{align*}
 t_{Q,\tau}[u,v] & :=-(Q,\overline{u}'v+\overline{u}v')_{L^{2}(\mathbb{R})}+(\tau,\overline{u}v)_{L^{2}(\mathbb{R})},\qquad
 & \mathrm{Dom}(t_{Q,\tau}) & :=\mathrm{Dom}(\mathrm{S}_{00}(q)), \\
 \dot{t}_{0}[u,v] & :=(u',v')_{L^{2}(\mathbb{R})},\qquad & \mathrm{Dom}(\dot{t}_{0}) & :=\mathrm{Dom}(\mathrm{S}_{00}(q)).
\end{align*}
Then due to Lemma~\ref{lm_Prf16} forms $t_{Q,\tau}$ are 0-bounded with respect to 
the  densely defined positive form $\dot{t}_{0}$:
\begin{align}\label{eq_Prf20}
 \left|t_{Q,\tau}[u] \right| & \leq K\varepsilon \dot{t}_{0}[u]+4K\varepsilon^{-3}\lVert u\rVert_{L^{2}(\mathbb{R})}^{2}
 \qquad \forall\varepsilon\in (0,1],\; u\in \mathrm{Dom}(\dot{t}_{0}), \\
 K & := 2\left(\lVert Q\rVert_{L_{unif}^{2}(\mathbb{R})}+\lVert\tau\rVert_{L_{unif}^{1}(\mathbb{R})}\right). \notag
\end{align}

Formula \eqref{eq_Prf20} implies that sesquilinear forms $\dot{t}_{\mathrm{S}_{00}}=\dot{t}_{0}+t_{Q,\tau}$ 
are closable, $t_{\mathrm{S}_{00}}:=(\dot{t}_{\mathrm{S}_{00}})^{\sim}$:
\begin{equation*}\label{eq_Prf30}
  t_{\mathrm{S}_{00}}[u,v]=(u',v')_{L^{2}(\mathbb{R})}-(Q,\overline{u}'v+\overline{u}v')_{L^{2}(\mathbb{R})}+
  (\tau,\overline{u}v)_{L^{2}(\mathbb{R})},\quad \mathrm{Dom}(t_{\mathrm{S}_{00}})=H^{1}(\mathbb{R}).
\end{equation*}
Forms $t_{\mathrm{S}_{00}}$ are densely defined, closed and sectorial. 
Then due to the First Representation Theorem \cite{Kt1995}, with the
sesqulinear forms $t_{\mathrm{S}_{00}}$ we associate $m$-sectorial operators
$\mathrm{S}_{F}(q)$ that are the Friedrichs extensions of operators $\mathrm{S}_{00}(q)$.

\begin{proposition}\label{pr_Prf14}
The $m$-sectorial operators $\mathrm{S}_{F}(q)$ are described in the following way:
\begin{equation*}
  \mathrm{S}_{F}u\equiv \mathrm{S}_{F}(q)u=\mathrm{l}[u],\qquad
   \mathrm{Dom}(\mathrm{S}_{F})=\left\{u\in H^{1}(\mathbb{R})\,\left|\,u,\,u^{[1]}\in
 \mathrm{AC}_{loc}(\mathbb{R}),\mathrm{l}[u]\in L^{2}(\mathbb{R})\right.\right\}.
\end{equation*}
\end{proposition}
The proof of Proposition~\ref{pr_Prf14} is similar to the proof of \cite[Theorem~3.5]{HrMk2001} 
for real-valued distributions $q\in H_{unif}^{-1}(\mathbb{R})$.

Thus we have established that the following relations hold:
\begin{equation}\label{eq_Prf32}
 \mathrm{S}_{00}\subset\mathrm{S}_{0}\subset\mathrm{S}_{F}\subset\mathrm{S}.
\end{equation}
Passing in to the adjoint operators \eqref{eq_Prf32} and using property~$2^{0}$ of Proposition~\ref{pr_Prf10}, 
we obtain:
\begin{equation}\label{eq_Prf34}
 \mathrm{S}_{00}^{+}\subset\mathrm{S}_{0}^{+}\subset\mathrm{S}_{F}^{\ast}\subset\mathrm{S}^{+}.
\end{equation}
One can easily prove that operators $\mathrm{S}_{F}^{\ast}$ coincide with Friedrichs extensions $\mathrm{S}_{F}^{+}$
of operators~$\mathrm{S}_{00}^{+}$.

Let us now define the operators \eqref{eq_10} as form-sums.

Consider the sesquilinear forms generated by the distributions $q\in H_{unif}^{-1}(\mathbb{R})$:
\begin{equation*}
 \dot{t}_{q}[u,v]:=\langle q(x)u,v\rangle, \qquad \mathrm{Dom}(\dot{t}_{q}):=C_{0}^{\infty}(\mathbb{R}),
\end{equation*}
where $\langle\cdot,\cdot\rangle$ is a sesquilinear form  pairing the spaces of generalized functions 
$\mathfrak{D}'(\mathbb{R})$ and test functions $C_{0}^{\infty}(\mathbb{R})$ with respect to the space 
$L^{2}(\mathbb{R})$.

Due to Lemma~\ref{lm_Prf16} for the forms
\begin{equation*}
 \dot{t}_{q}[u]=\langle q(x)u,v\rangle=-(Q,\overline{u}'u+\overline{u}u')_{L^{2}(\mathbb{R})}
+(\tau,|u|^{2})_{L^{2}(\mathbb{R})}, \qquad u\in C_{0}^{\infty}(\mathbb{R}),
\end{equation*}
the following estimates hold:
\begin{equation*}
 \left| \dot{t}_{q}[u]\right|\leq 2\left(\lVert Q\rVert_{L_{unif}^{2}(\mathbb{R})}+ \lVert
\tau\rVert_{L_{unif}^{1}(\mathbb{R})}\right) \left(\lVert u'\rVert_{L^{2}(\mathbb{R})}^{2}+4\lVert
u\rVert_{L^{2}(\mathbb{R})}^{2}\right), \qquad u\in C_{0}^{\infty}(\mathbb{R}).
\end{equation*}
Therefore forms $\dot{t}_{q}$ allow a continuous extension onto the space $H^{1}(\mathbb{R})$ \cite{NmSh1999}. 
The sesquilinear forms $\dot{t}_{q}[u,v]$ on the space $H^{1}(\mathbb{R})$ are represented as:
\begin{equation}\label{eq_Prf36}
 t_{q}[u,v]=-(Q,\overline{u}'v+\overline{u}v')_{L^{2}(\mathbb{R})}+(\tau,\overline{u}v)_{L^{2}(\mathbb{R})}, \qquad
 \mathrm{Dom}(t_{q})=H^{1}(\mathbb{R}).
\end{equation}
One may easily see that the following Lemma is true applying the estimates of Lemma~\ref{lm_Prf16}.
\begin{lemma}\label{lm_Prf18}
The sesquilinear forms $t_{q}$ are 0-bounded with respect to the sesquilinear form
\begin{equation*}
 t_{0}[u,v]:=(u',v')_{L^{2}(\mathbb{R})},\qquad \mathrm{Dom}(t_{0}):=H^{1}(\mathbb{R}).
\end{equation*}
\end{lemma}

Thus, the sesquilinear forms
\begin{equation}\label{eq_Prf38}
 t[u,v]:=t_{0}[u,v]+t_{q}[u,v],\qquad \mathrm{Dom}(t):=H^{1}(\mathbb{R}),
\end{equation}
are densely defined, closed and sectorial.
According to the First Representation Theorem \cite{Kt1995} with the forms $t$ one can associate 
$m$-sectorial operators $\mathrm{S}_{fs}(q)$, which are called the \textit{form-sums}
and denoted by:
\begin{align*}
 \mathrm{S}_{fs}\equiv\mathrm{S}_{fs}(q) & :=-\dfrac{d^{2}}{dx^{2}}\dotplus q(x), \\
 \mathrm{Dom}(\mathrm{S}_{fs}(q)) & :=\left\lbrace u\in H^{1}(\mathbb{R})\left| -u''+q(x)u\in
 L^{2}(\mathbb{R})\right.\right\rbrace.
\end{align*}

Since the forms $t$ coincide with the forms $t_{\mathrm{S}_{00}}$, the form-sum operators $\mathrm{S}_{fs}(q)$ and
the Friedrichs extensions $\mathrm{S}_{F}(q)$ of operators $\mathrm{S}_{00}(q)$ coincide: $\mathrm{S}_{F}(q)=\mathrm{S}_{fs}(q)$.

Thus, relations \eqref{eq_Prf32} and \eqref{eq_Prf34} take the following form:
\begin{align}
 \mathrm{S}_{00} & \subset\mathrm{S}_{0}\subset\mathrm{S}_{F}=\mathrm{S}_{fs}\subset\mathrm{S},\quad
 & \mathrm{Dom}(\mathrm{S}_{F})  \subset H^{1}(\mathbb{R}),\;
 \mathrm{Dom}(\mathrm{S})\subset H_{loc}^{1}(\mathbb{R}),\label{eq_Prf40.1} \\
 \mathrm{S}_{00}^{+} & \subset\mathrm{S}_{0}^{+}\subset \mathrm{S}_{F}^{+}=\mathrm{S}_{fs}^{+}
 =\mathrm{S}_{F}^{\ast}=\mathrm{S}_{fs}^{\ast} \subset\mathrm{S}^{+},\quad
 & \mathrm{Dom}(\mathrm{S}_{F}^{+})  \subset H^{1}(\mathbb{R}),\;
 \mathrm{Dom}(\mathrm{S}^{+})\subset H_{loc}^{1}(\mathbb{R}).\label{eq_Prf40.2}
\end{align}

\begin{proposition}\label{pr_Prf16}
Suppose $\mathrm{Dom}(\mathrm{S})\subset H^{1}(\mathbb{R})$. 
Then operators $\mathrm{S}_{0}(q)$ and $\mathrm{S}_{0}^{+}(q)$ are $m$-sectorial and
\begin{align*}
 & \mathrm{S}_{0}=\mathrm{S}_{F}=\mathrm{S}_{fs}=\mathrm{S}, \\ % \label{eq_Prf42.1}
 & \mathrm{S}_{0}^{+}=\mathrm{S}_{F}^{+}=\mathrm{S}_{fs}^{+}=\mathrm{S}_{F}^{\ast}=\mathrm{S}_{fs}^{\ast}
 =\mathrm{S}^{+}. % \label{eq_Prf42.2}
\end{align*}
\end{proposition}
\begin{proof}
Let the assumptions of Proposition~\ref{pr_Prf16} be fulfilled.
Then due to property~$5^{0}$ of Proposition~\ref{pr_Prf10} we also have 
$\mathrm{Dom}(\mathrm{S}^{+})\subset H^{1}(\mathbb{R})$.

For $Q\in L_{unif}^{2}(\mathbb{R})$ and $u\in H^{1}(\mathbb{R})$ we have $Qu\in L^{2}(\mathbb{R})$ 
\cite[Theorem~3.5]{HrMk2001} and therefore
\begin{align*}
 u^{[1]}=u'-Qu\in L^{2}(\mathbb{R}),\qquad
 v^{\{1\}}=v'-\overline{Q}v\in L^{2}(\mathbb{R}).
\end{align*}
So, $\forall u\in \mathrm{Dom}(\mathrm{S})$ and $\forall v\in \mathrm{Dom}(\mathrm{S}^{+})$ we obtain:
\begin{align}
 [u,v](-\infty) & =\lim_{t\rightarrow -\infty}[u,v](t)=\lim_{t\rightarrow
-\infty}\left(u(t)\overline{v^{\{1\}}(t)}-u^{[1]}(t)\overline{v(t)}\right)=0, \label{eq_Prf44.1} \\
 [u,v](\infty) & =\lim_{t\rightarrow \infty}[u,v](t)=\lim_{t\rightarrow
 \infty}\left(u(t)\overline{v^{\{1\}}(t)}-u^{[1]}(t)\overline{v(t)}\right)=0. \label{eq_Prf44.2}
\end{align}
Taking into account \eqref{eq_Prf44.1} and \eqref{eq_Prf44.2}, property~$4^{0}$ of Proposition~\ref{pr_Prf10} implies 
the equalities:
\begin{equation*}
 \mathrm{S}_{0}=\mathrm{S}_{F}=\mathrm{S},\qquad \mathrm{S}_{0}^{+}=\mathrm{S}_{F}^{+}=\mathrm{S}^{+}.
\end{equation*}

Proposition is proved.
\end{proof}

Due to Proposition~\ref{pr_Prf12} (see Section 3.3~\ref{ssec:PrfC}) operators $\mathrm{S}_{0}(q)$ are quasiaccretive:
\begin{equation*}
 \mathrm{Re}\,(\mathrm{S}_{0}u,u)_{L^{2}(\mathbb{R})}\geq -4\left(2K+1\right)^{4}\lVert
 u\rVert_{L^{2}(\mathbb{R})}^{2},\qquad u\in \mathrm{Dom}(\mathrm{S}_{0}).
\end{equation*}
In what follows w.l.a.g. we assume that
\begin{equation}\label{eq_Prf50}
 \mathrm{Re}\,(\mathrm{S}_{0}u,u)_{L^{2}(\mathbb{R})}\geq \lVert u\rVert_{L^{2}(\mathbb{R})}^{2},\qquad u\in
\mathrm{Dom}(\mathrm{S}_{0}).
\end{equation}
Obviously, together with \eqref{eq_Prf50} the following is also true:
\begin{equation}\label{eq_Prf52}
 \mathrm{Re}\,(\mathrm{S}_{0}^{+}v,v)_{L^{2}(\mathbb{R})}\geq \lVert v\rVert_{L^{2}(\mathbb{R})}^{2},\qquad v\in
 \mathrm{Dom}(\mathrm{S}_{0}^{+}).
\end{equation}
Indeed, taking into account the property~$5^{0}$ of Proposition~\ref{pr_Prf10} we get:
\begin{equation*}
 \mathrm{Re}\,(\mathrm{S}_{0}^{+}v,v)_{L^{2}(\mathbb{R})}
 =\mathrm{Re}\,\overline{(\mathrm{S}_{0}^{+}v,v)}_{L^{2}(\mathbb{R})}
 =\mathrm{Re}\,(\mathrm{S}_{0}\overline{v},\overline{v})_{L^{2}(\mathbb{R})}
 \geq  \lVert v\rVert_{L^{2}(\mathbb{R})}^{2}\qquad \forall v\in \mathrm{Dom}(\mathrm{S}_{0}^{+}).
\end{equation*}

The following lemma is used in the proof of Theorem~\ref{thMn_A}. 
It is proved by direct calculation.

\begin{lemma}\label{lm_Prf20}
Suppose $u\in\mathrm{Dom}(\mathrm{S})$. 
Then $\forall\varphi\in C_{0}^{\infty}(\mathbb{R})$:
\begin{equation*}
 i)\;\mathrm{l}[\varphi
 u]=\varphi\mathrm{l}[u]-\varphi''u-2\varphi'u';\qquad ii)\; \varphi u\in \mathrm{Dom}(\mathrm{S}_{00}). \hspace{150pt}
\end{equation*}
\end{lemma}

Now let us prove Theorem~\ref{thMn_A}.

Let us prove that operators $\mathrm{S}_{0}(q)$ are quasi-$m$-accretive. 
It is sufficient to show that 
\begin{equation*}
 \mathrm{def}\,\mathrm{S}_{0}(q):=\mathrm{dim}\,(\mathrm{ran}\,\mathrm{S}_{0}(q))^{\bot}\equiv\mathrm{dim}\,(\ker\,\mathrm{S}^{+}(q))=0.
\end{equation*}

Let $v(x)$ be a solution of the equation 
\begin{equation}\label{eq_Prf54}
 \mathrm{S}^{+}(q)v=0.
\end{equation}
Let us show that $v(x)\equiv 0$.

For any real function $\varphi\in C_{0}^{\infty}(\mathbb{R})$ due to Lemma~\ref{lm_Prf20} and
property~$5^{0}$ of Proposition~\ref{pr_Prf10} we have $\varphi v\in \mathrm{Dom}(\mathrm{S}_{00}^{+})$. 
Therefore, taking into consideration that $\mathrm{l}^{+}[v]=0$ due to \eqref{eq_Prf54}, 
one calculates:
\begin{equation}\label{eq_Prf56}
 (\mathrm{S}^{+}\varphi v,\varphi v)_{L^{2}(\mathbb{R})}=\int_{\mathbb{R}}(\varphi')^{2}|v|^{2}d\,x
 +\int_{\mathbb{R}}\varphi\varphi'(v\overline{v}'-v'\overline{v})d\,x.
\end{equation}
Considering \eqref{eq_Prf50} and that 
\begin{equation*}
 \mathrm{Re}\,\int_{\mathbb{R}}\varphi\varphi'(v\overline{v}'-v'\overline{v})d\,x=0,
\end{equation*}
from \eqref{eq_Prf56} we obtain:
\begin{equation}\label{eq_Prf58}
 \int_{\mathbb{R}}(\varphi')^{2}|v|^{2}d\,x\geq \int_{\mathbb{R}}\varphi^{2}|v|^{2}d\,x\qquad \forall \varphi\in
 C_{0}^{\infty}(\mathbb{R}),\; \mathrm{Im}\,\varphi=0.
\end{equation}
Let us then take a sequence of functions $\{\varphi_{n}\}_{n\in \mathbb{N}}$ such that:
\begin{itemize}
 \item [i)] $\varphi_{n}\in C_{0}^{\infty}(\mathbb{R})$, $\mathrm{Im}\,\varphi_{n}\equiv0$;
 \item [ii)] $\mathrm{supp}\,\varphi_{n}\subset [-n-1,n+1]$;
 \item [iii)] $\varphi_{n}(x)=1$, $x\in [-n,n]$;
 \item [iv)] $|\varphi_{n}'(x)|\leq C$.
\end{itemize}
Substituting functions $\varphi_{n}$ into \eqref{eq_Prf58} we receive
\begin{equation*}
 \int_{-n}^{n}|v|^{2}d\,x\leq \int_{\mathbb{R}}\varphi_{n}^{2}|v|^{2}d\,x\leq \int_{\mathbb{R}}(\varphi_{n}')^{2}|v|^{2}d\,x
 \leq C^{2}\int_{n\leq|x|\leq n+1}\limits|v|^{2}d\,x,
\end{equation*}
that is
\begin{equation}\label{eq_Prf60}
 \int_{-n}^{n}|v|^{2}d\,x\leq C^{2}\int_{n\leq|x|\leq n+1}\limits|v|^{2}d\,x.
\end{equation}
Taking into account that $v(x)\in L^{2}(\mathbb{R})$, passing in \eqref{eq_Prf60} to the limit as $n\rightarrow\infty$
we obtain $v(x)\equiv 0$.

Thus, operators $\mathrm{S}_{0}(q)$ are proved to be quasi-$m$-accretive. 
Due to Proposition~\ref{pr_Prf12} they are $m$-sectorial.

Therefore, by the properties of the Friedrichs extensions \cite{Kt1995} we have:
\begin{equation}\label{eq_Prf62}
 \mathrm{S}_{0}(q)=\mathrm{S}_{F}(q).
\end{equation}
Then taking into account property~$2^{0}$ of Proposition~\ref{pr_Prf10} from \eqref{eq_Prf62} we derive:
\begin{equation*}
 \mathrm{S}^{+}(q)=\mathrm{S}_{F}^{+}(q),\qquad \mathrm{Dom}(\mathrm{S}^{+}(q))\subset H^{1}(\mathbb{R}).
\end{equation*}
Due to the property~$5^{0}$ of Proposition~\ref{pr_Prf10} from Proposition~\ref{pr_Prf16} we finally get necessary result
\begin{equation*}
 \mathrm{S}_{0}=\mathrm{S}_{F}=\mathrm{S}_{fs}=\mathrm{S}.
\end{equation*}

Theorem~\ref{thMn_A} is proved completely. \hfill{$\square$}

%%%%%%%%%%%%%%%%%%%%%%%%%%%%%%%%%%%%%%%%%%%%%%%%%%%%%%%%%%%%%%%%%%%%%%%%%%%%%%%%%%%%%%%%%%%%%%%%%%%%%%%%%%%%%%%%%%%%%%%%%%%%

\subsection{Proof of Theorem~\ref{thMn_B}}\label{ssec:PrfB}
Let us suppose that the assumptions of theorem, that is the formula \eqref{eq_22} (or equivalently \eqref{eq_24}), hold. 
Consider the sesquilinear forms 
\begin{align*}
 \dot{t}_{0}[u,v] & :=(\mathrm{S}(q)u,v)_{L^{2}(\mathbb{R})},\qquad
 \mathrm{Dom}(\dot{t}_{0}):=\mathrm{Dom}(\mathrm{S}(q)), \\
  \dot{t}_{n}[u,v] & :=(\mathrm{S}(q_{n})u,v)_{L^{2}(\mathbb{R})},\qquad
 \mathrm{Dom}(\dot{t}_{n}):=\mathrm{Dom}(\mathrm{S}(q_{n})),\;n\in \mathbb{N}.
\end{align*}
The forms $\dot{t}_{0}$ and $\dot{t}_{n}$, $n\in \mathbb{N}$, are densely defined, closable and sectorial. 
Their closures may be represented in the following way:
\begin{align*}
 t_{0}[u,v] & =(u',v')_{L^{2}(\mathbb{R})}-(Q,\overline{u}'v+\overline{u}v')_{L^{2}(\mathbb{R})}+
 (\tau,\overline{u}v)_{L^{2}(\mathbb{R})}, \qquad \mathrm{Dom}(t_{0})=H^{1}(\mathbb{R}), \\
  t_{n}[u,v] & =(u',v')_{L^{2}(\mathbb{R})}-(Q_{n},\overline{u}'v+\overline{u}v')_{L^{2}(\mathbb{R})}+
 (\tau_{n},\overline{u}v)_{L^{2}(\mathbb{R})}, \qquad \mathrm{Dom}(t_{n})=H^{1}(\mathbb{R}).
\end{align*}

Further, applying the estimates of Lemma~\ref{lm_Prf16} we get:
\begin{equation}\label{eq_Prf100}
 \left|t_{n}[u]-t_{0}[u]\right|\leq a_{n}\lVert u'\rVert_{L^{2}(\mathbb{R})}^{2}+4a_{n}\lVert
 u\rVert_{L^{2}(\mathbb{R})}^{2},
\end{equation}
where
\begin{equation*}
 a_{n}:=2\left(\lVert Q-Q_{n}\rVert_{L_{unif}^{2}(\mathbb{R})}+\lVert
 \tau-\tau_{n}\rVert_{L_{unif}^{1}(\mathbb{R})}\right),
\end{equation*}
and similarly to the proof of Lemma~\ref{lm_SsqlFm10} (see below) we obtain:
\begin{equation}\label{eq_Prf102}
 2\mathrm{Re}\,t_{0}[u]+4\lVert u\rVert_{L^{2}(\mathbb{R})}^{2}\geq \lVert u'\rVert_{L^{2}(\mathbb{R})}^{2}.
\end{equation}
Formulas \eqref{eq_Prf100} and \eqref{eq_Prf102} together with \eqref{eq_22}, \eqref{eq_24} imply:
\begin{equation*}\label{eq_Prf104}
 \left|t_{n}[u]-t_{0}[u]\right|\leq 2a_{n}\mathrm{Re}\,t_{0}[u]+8a_{n}\lVert
 u\rVert_{L^{2}(\mathbb{R})}^{2},\qquad a_{n}\rightarrow 0,\; n\rightarrow\infty.
\end{equation*}

To complete the proof we only need to apply \cite[Theorem~VI.3.6]{Kt1995}.

Theorem~\ref{thMn_B} is proved completely. \hfill{$\square$}

To prove Corollary~\ref{cr_thMnB_B1} we need in an auxiliary result. It has an independent interest also.
\begin{theorem}\label{th_StpSp10}
The set
\begin{equation}\label{eq_StpSp10}
 C^{\infty}(\mathbb{R})\cap L_{unif}^{p}(\mathbb{R})
\end{equation}
is everywhere dense in the Stepanov space $L_{unif}^{p}(\mathbb{R})$, $1\leq p<\infty$.
\end{theorem}
\begin{proof}
Set for $f\in L_{loc}^{p}(\mathbb{R})$, $f_{n}:=\chi_{[n,n+1)}f$, $n\in \mathbb{Z}$. 
Let $\varepsilon>0$ be given. 
Since the set $C_{0}^{\infty}(a,b)$ is dense in the space $L^{p}(a,b)$, there is a function sequence 
$g_{n}\in C_{0}^{\infty}(\mathbb{R})$, $\mathrm{supp}\,g_{n}\subset (n,n+1)$, such that 
$\lVert f_{n}-g_{n}\rVert_{L^{p}(\mathbb{R})}<\varepsilon 2^{-|n|-2}$. 
Set $g_{\varepsilon}:=\sum_{n\in \mathbb{Z}}g_{n}$. Then $g_{\varepsilon}\in C^{\infty}(\mathbb{R})$ and 
$\lVert f-g_{\varepsilon}\rVert_{L^{p}(\mathbb{R})}<\varepsilon$. 
If $f\in L_{unif}^{p}(\mathbb{R})$, then the function $g_{\varepsilon}\in L_{unif}^{p}(\mathbb{R})$ 
since$\lVert f-g_{\varepsilon}\rVert_{L_{unif}^{p}(\mathbb{R})}<\varepsilon$.
If the function $f$ is real-valued, then so are the functions $f_n$ as well.
Therefore, the functions $ g_{\varepsilon} $ may be chosen to be real-valued.
\end{proof}

Theorem~\ref{th_StpSp10} and \cite[Theorem~2.1]{HrMk2001} imply the following important statement.
\begin{corollaryTh}\label{cr_StpSp10}
The set
\begin{equation*}\label{eq_StpSp28}
 C^{\infty}(\mathbb{R})\cap L_{unif}^{1}(\mathbb{R})
\end{equation*}
is everywhere dense in the space $H_{unif}^{-1}(\mathbb{R})$.
\end{corollaryTh}

Then Corollary~\ref{cr_thMnB_B1} follows from Theorem~\ref{thMn_B} and Corollary~\ref{cr_StpSp10}.

%%%%%%%%%%%%%%%%%%%%%%%%%%%%%%%%%%%%%%%%%%%%%%%%%%%%%%%%%%%%%%%%%%%%%%%%%%%%%%%%%%%%%%%%%%%%%%%%%%%%%%%%%%%%%%%%%%%%%%%%%%%%
%%%%%%%%%%%%%%%%%%%%%%%%%%%%%%%%%%%%%%%%%%%%%%%%%%%%%%%%%%%%%%%%%%%%%%%%%%%%%%%%%%%%%%%%%%%%%%%%%%%%%%%%%%%%%%%%%%%%%%%%%%%%

\subsection{Proof of Theorem~\ref{thMn_C}}\label{ssec:PrfC}
Theorem~\ref{thMn_C} follows from Theorem~\ref{th_SsqlFm10} below regarding perturbations of 
a positive quadratic form. It is abstract and can be of independent interest.

Let in an abstract Hilbert space $H$ a densely defined closed positive sesquilinear form $\alpha_{0}[u,v]$ with domain 
$\mathrm{Dom}(\alpha_{0})\subset H$ be given. Let $\beta[u,v]$ be a sesquilinear form defined on $H$
with a domain $\mathrm{Dom}(\beta)\supset\mathrm{Dom}(\alpha_{0})$.

Suppose the form $\beta$ satisfies the following estimate:
\begin{equation}\label{eq_SsqlFmmc}
  \exists\, a,b,s>0:\qquad \lvert\beta[u]\rvert\leq a\varepsilon \alpha_{0}[u]+b\varepsilon^{-s}\lVert u\rVert_{H}^{2}\qquad
\forall
 \varepsilon>0,\;u\in \mathrm{Dom}(\alpha_{0}).
\end{equation}

Consider on the Hilbert space $H$ the sum of forms $\alpha_{0}$ and $\beta$:
\begin{equation*}
 \alpha[u,v]:=\alpha_{0}[u,v]+\beta[u,v],\qquad \mathrm{Dom}(\alpha):=\mathrm{Dom}(\alpha_{0}).
\end{equation*}
A sesquilinear form $\alpha$ is densely defined closed and sectorial form on the Hilbert space $H$. 
Let $\Theta(\alpha)$ be a numerical range of $\alpha$:
\begin{equation*}
 \Theta(\alpha):=\alpha[u], \qquad u\in \mathrm{Dom}(\alpha), \;\lVert u\rVert_{H}=1.
\end{equation*}

According to our assumptions $\Theta(\alpha_{0})\subset [0,\infty)$. Let us find the properties of the set $\Theta(\alpha)$.
To do that we require the following two lemmas. 
\begin{lemma}\label{lm_SsqlFm10}
The following estimates hold:
\begin{equation}
 \left|\mathrm{Im}\,\alpha[u]\right|\leq 2a\varepsilon\mathrm{Re}\,\alpha[u]+2b\varepsilon^{-s}\lVert u\rVert_{H}^{2}, \qquad
 0<\varepsilon\leq (2a+1)^{-1}.
\end{equation}
\end{lemma}
\begin{proof}
According to our assumptions we have:
\begin{align*}
 \mathrm{Re}\,\alpha[u] & =\alpha_{0}[u]+\mathrm{Re}\, \beta[u],\qquad \mathrm{Im}\,\alpha[u] =\mathrm{Im}\,\beta[u],
\end{align*}
and due to \eqref{eq_SsqlFmmc}:
\begin{equation}\label{eq_SsqlFm10}
 \lvert\mathrm{Im}\,\alpha[u]\rvert\leq a\varepsilon \alpha_{0}[u]+b\varepsilon^{-s}\lVert u\rVert_{H}^{2}.
\end{equation}
Furthermore given that $0<\varepsilon\leq (2a+1)^{-1}$ and therefore $1-a\varepsilon\geq\frac{1}{2}$
we have for $\mathrm{Re}\,\alpha[u]$:
\begin{equation*}\label{eq_SsqlFm12}
 \mathrm{Re}\,\alpha[u]\geq \alpha_{0}[u]-\left|\mathrm{Re}\,\beta[u]\right|\geq 
 (1-a\varepsilon)\alpha_{0}[u]-b\varepsilon^{-s}\lVert
 u\rVert_{H}^{2}\geq \frac{1}{2}\alpha_{0}[u]-b\varepsilon^{-s}\lVert u\rVert_{H}^{2},
\end{equation*}
and
\begin{align}
 2a\varepsilon \mathrm{Re}\,\alpha[u] & \geq a\varepsilon \alpha_{0}[u]-2a\varepsilon\cdotp 
 b\varepsilon^{-s}\lVert u\rVert_{H}^{2}
 \geq a\varepsilon \alpha_{0}[u]-b\varepsilon^{-s}\lVert u\rVert_{H}^{2}, \notag \\
 2a\varepsilon \mathrm{Re}\,\alpha[u]+b\varepsilon^{-s}\lVert u\rVert_{H}^{2} & \geq a\varepsilon \alpha_{0}[u].
 \label{eq_SsqlFm14}
\end{align}
From \eqref{eq_SsqlFm10} and \eqref{eq_SsqlFm14} we receive the required estimates:
\begin{equation*}
 \left|\mathrm{Im}\,\alpha[u]\right|\leq 2a\varepsilon\mathrm{Re}\,\alpha[u]+2b\varepsilon^{-s}\lVert u\rVert_{H}^{2}.
\end{equation*}

Lemma is proved.
\end{proof}

We introduce the following notation:
\begin{align*}
 \mathcal{S}_{a,b,s,\varepsilon} & :=\left\lbrace\lambda\in \mathbb{C}\left|\left|\mathrm{Im}\,\lambda\right|\leq
 2a\varepsilon\mathrm{Re}\,\lambda+2b\varepsilon^{-s}  \right.\right\rbrace,  \\
 \mathcal{M}_{a,b,s} & := \bigcap_{0<\varepsilon\leq (2a+1)^{-1}}\mathcal{S}_{a,b,s,\varepsilon}.
\end{align*}
Then due to Lemma~\ref{lm_SsqlFm10} we have $\Theta(\alpha)\subset \mathcal{M}_{a,b,s}$.
\begin{lemma}\label{lm_SsqlFm12}
The set $\mathcal{M}_{a,b,s}$ can be written as:
\begin{equation*}
  \mathcal{M}_{a,b,s}=
  \begin{cases}
   \left\lbrace\lambda\in \mathbb{C}\left| \left|\mathrm{Im}\,\lambda\right|\leq
    \dfrac{2a}{2a+1}\mathrm{Re}\,\lambda+2b(2a+1)^{s}  \right.\right\rbrace,\;
    \lambda_{0}\leq\mathrm{Re}\,\lambda\leq \lambda_{1}, \vspace{5pt}  \\
    \left\lbrace\lambda\in \mathbb{C}\left|\left|\mathrm{Im}\,\lambda\right|\leq
    2(s+1)b^{1/(s+1)}\left(\dfrac{a}{s}\right)^{s/(s+1)}\left(\mathrm{Re}\,\lambda\right)^{s/(s+1)}\right.\right\rbrace,\;
    \lambda_{1}<\mathrm{Re}\,\lambda,
   \end{cases}
\end{equation*}
where $\lambda_{0}:=-\dfrac{b}{a}(2a+1)^{s+1}$ is the vertex of sector 
\begin{equation*}
 \left\lbrace\lambda\in
\mathbb{C}\left|\left|\mathrm{Im}\,\lambda\right|
\leq \frac{2a}{2a+1}\mathrm{Re}\,\lambda+2b(2a+1)^{s} \right.\right\rbrace,
\end{equation*}
and $\lambda_{1}:=\dfrac{bs}{a}(2a+1)^{s+1}$.
\end{lemma}
\begin{proof}
For convenience we will find the description of the set $\mathcal{M}_{a,b,s}$ in $\mathbb{R}^{2}$.

Let
\begin{equation}\label{eq_SsqlFm16}
  y=2a\varepsilon x +2b\varepsilon^{-s}
\end{equation}
be the line which bounds the corresponding sector from above.
Let us find the locus of points of intersection of these lines when
$0<\varepsilon<(2a+1)^{-1}$:
\begin{align*}
  & 2a\varepsilon_{1}x+2b\varepsilon_{1}^{-s}=2a\varepsilon_{2}x+2b\varepsilon_{2}^{-s}, \\
  & 2a(\varepsilon_{1}-\varepsilon_{2})x=2b(\varepsilon_{2}^{-s}-\varepsilon_{1}^{-s}), \\
  & x=-\dfrac{b}{a}\cdotp\dfrac{\varepsilon_{1}^{-s}-\varepsilon_{2}^{-s}}{\varepsilon_{1}-\varepsilon_{2}}, \\
  & x\underset{\varepsilon_{2}\rightarrow\varepsilon_{1}}{\longrightarrow}\dfrac{sb}{a}\varepsilon_{1}^{-s-1},
\end{align*}
and
\begin{equation*}
  y=2a\varepsilon_{1}\cdotp\dfrac{sb}{a}\varepsilon_{1}^{-s-1}+2b\varepsilon_{1}^{-s}=2(s+1)b\varepsilon_{1}^{-s}.
\end{equation*}
So, for $\varepsilon=(2a+1)^{-1}$ the set $\mathcal{M}_{a,b,s}$ is bounded from above by the line:
\begin{equation}\label{eq_SsqlFm18}
  y=\frac{2a}{2a+1} x +2b(2a+1)^{s},
\end{equation}
and for $0<\varepsilon<(2a+1)^{-1}$ by the curves:
\begin{equation}\label{eq_SsqlFm20.1}
  \left\{
  \begin{array}{l}
    x=\frac{sb}{a}\varepsilon^{-s-1}, \\
    y=2(s+1)b\varepsilon^{-s}. \
  \end{array}
  \right.
\end{equation}
If we express $\varepsilon$ through $x$ in the first equality of~\eqref{eq_SsqlFm20.1}
and substitute it in the equality for $y$, we obtain an explicit equation for curves \eqref{eq_SsqlFm20.1}:
\begin{equation}\label{eq_SsqlFm20.2}
  y=2(s+1)b^{1/(s+1)}\left(\dfrac{a}{s}\right)^{s/(s+1)}x^{s/(s+1)},\quad x>x_{1},\; x_{1}:=\dfrac{bs}{a}(2a+1)^{s+1}.
\end{equation}

The set $\mathcal{M}_{a,b,s}$ is bounded below by curves of the form \eqref{eq_SsqlFm18} and \eqref{eq_SsqlFm20.2} 
with $-y$ instead of $y$.

Thus the set $\mathcal{M}_{a,b,s}$ in $\mathbb{R}^{2}$ may be represented in the following way:
\begin{equation*}
  \mathcal{M}_{a,b,s}=
  \begin{cases}
   \left\lbrace (x,y)\in \mathbb{R}^{2}\left| \left|y\right|\leq \dfrac{2a}{2a+1}x+2b(2a+1)^{s}
   \right.\right\rbrace,\quad & x_{0}\leq x\leq x_{1}, \vspace{5pt} \\
   \left\lbrace (x,y)\in \mathbb{R}^{2}\left|\left|y\right|\leq 2(s+1)b^{1/(s+1)}\left(\dfrac{a}{s}\right)^{s/(s+1)}
   x^{s/(s+1)}\right.\right\rbrace,\quad &  x_{1}<x,
   \end{cases}
\end{equation*}
where $x_{0}:=-\dfrac{b}{a}(2a+1)^{s+1}$ is the vertex of sector 
\begin{equation*}
 \left\lbrace (x,y)\in \mathbb{R}^{2}\left| \left|y\right|\leq
\dfrac{2a}{2a+1}x+2b(2a+1)^{s}\right.\right\rbrace,
\end{equation*}
and $x_{1}=\dfrac{bs}{a}(2a+1)^{s+1}$.

Lemma is proved.
\end{proof}

Lemmas~\ref{lm_SsqlFm10} and \ref{lm_SsqlFm12} imply the following theorem.
\begin{theorem}\label{th_SsqlFm10}
The numerical range $\Theta(\alpha)$ of the sesquilinear form $\alpha$ is a subset of the set $\mathcal{M}_{a,b,s}$:
\begin{equation*}
  \mathcal{M}_{a,b,s}=
  \begin{cases}
   \left\lbrace\lambda\in \mathbb{C}\left| \left|\mathrm{Im}\,\lambda\right|\leq
    \dfrac{2a}{2a+1}\mathrm{Re}\,\lambda+2b(2a+1)^{s}  \right.\right\rbrace,\;
    \lambda_{0}\leq\mathrm{Re}\,\lambda\leq \lambda_{1}, \vspace{5pt}  \\
    \left\lbrace\lambda\in \mathbb{C}\left|\left|\mathrm{Im}\,\lambda\right|\leq
    2(s+1)b^{1/(s+1)}\left(\dfrac{a}{s}\right)^{s/(s+1)}\left(\mathrm{Re}\,\lambda\right)^{s/(s+1)}\right.\right\rbrace,\;
    \lambda_{1}<\mathrm{Re}\,\lambda,
   \end{cases}
\end{equation*}
where $\lambda_{0}=-\dfrac{b}{a}(2a+1)^{s+1}$ is the vertex of sector 
\begin{equation*}
 \left\lbrace\lambda\in
\mathbb{C}\left|\left|\mathrm{Im}\,\lambda\right|
\leq \frac{2a}{2a+1}\mathrm{Re}\,\lambda+2b(2a+1)^{s} \right.\right\rbrace,
\end{equation*}
and $\lambda_{1}=\dfrac{bs}{a}(2a+1)^{s+1}$.
\end{theorem}

\begin{remarkTh}\label{rm_SsqlFm10}
Direct calculations show that the following inclusion is valid:
\begin{equation*}\label{eq_SsqlFm22}
 \mathcal{M}_{a,b,s}\subset \left\lbrace\lambda\in \mathbb{C}\left| \left|\mathrm{Im}\,\lambda\right|\leq
 2(s+1)b^{1/(s+1)}\left(\dfrac{a}{s}\right)^{s/(s+1)}\left(\mathrm{Re}\,\lambda+\dfrac{b}{a}(2a+1)^{s+1}\right)^{s/(s+1)}
 \right.\right\rbrace.
\end{equation*}
\end{remarkTh}

Theorem~\ref{th_SsqlFm10} is useful for preliminary localisation of a spectrum of various operators.

For instance, if the potential $q\in H_{unif}^{-1}(\mathbb{R})$ is a complex-valued regular Borel measure 
such that:
\begin{equation*}
 q=Q',\quad Q\in \mathrm{BV}_{loc}(\mathbb{R}):\quad \left|q(I)\right|\equiv \left|\int_{I}d\,Q\right|\leq K_{0}, 
 \quad K_{0}>0,
\end{equation*}
for any interval $I\subset \mathbb{R}$ of a unit length, then forms satisfy the estimates \eqref{eq_SsqlFmmc} with 
$a=b=4K_{0}$, $s=1$ \cite{Tip1990}:
\begin{equation*}
 \left|t_{q}[u]\right|\equiv\left|\int_{I}|u|^{2}d\,Q\right|\leq 4K_{0}\varepsilon \lVert
 u'\rVert_{L^{2}(\mathbb{R})}^{2}+4K_{0}\varepsilon^{-1} \lVert u\rVert_{L^{2}(\mathbb{R})}^{2}\qquad \forall\varepsilon\in (0,1]
,\;
 u\in H^{1}(\mathbb{R}).
\end{equation*}
Then due to Theorem~\ref{th_SsqlFm10} the spectra $\mathrm{spec}(\mathrm{S}(q))$ of operators $\mathrm{S}(q)$ 
belong to a quadratic parabola:
\begin{equation}\label{eq_Tip}
 \mathrm{spec}(S(q))\subset \left\lbrace\lambda\in \mathbb{C}\left| \left|\mathrm{Im}\,\lambda\right|\leq
 16K_{0}\left(\mathrm{Re}\,\lambda+(8K_{0}+1)^{2}\right)^{1/2} \right.\right\rbrace,
\end{equation}
compare with \cite[Proposition~2.3]{Tip1990}.

Applying Theorem~\ref{th_SsqlFm10} and estimates \eqref{eq_Prf20} 
we obtain a description of the numerical ranges of preminimal operators 
$\mathrm{S}_{00}(q)$ and $\mathrm{S}_{00}^{+}(q)$.

\begin{proposition}\label{pr_Prf12}
Operators $\mathrm{S}_{00}(q)$ and $\mathrm{S}_{00}^{+}(q)$ are sectorial: 
for arbitrary $\varepsilon>0$ numerical ranges $\Theta(\mathrm{S}_{00}(q))$ and $\Theta(\mathrm{S}_{00}^{+}(q))$ 
are located within the sector:
\begin{align*}
 \mathcal{S}_{K,\varepsilon} & :=\left\{\lambda\in \mathbb{C}\left| \left\lvert \mathrm{Im}\,\lambda\right\rvert\leq
  2K\varepsilon\,\mathrm{Re}\,\lambda+8K\varepsilon^{-3} \right.\right\},\qquad 0<\varepsilon\leq (2K+1)^{-1}.
\end{align*}
Furthermore
\begin{equation*}
 \Theta(\mathrm{S}_{00}(q))\subset \mathcal{M}_{K},\qquad \Theta(\mathrm{S}_{00}^{+}(q))\subset \mathcal{M}_{K},
\end{equation*}
where
\begin{equation*}
 \mathcal{M}_{K}:=
 \begin{cases}
 \left\{\lambda\in \mathbb{C}\left| \left\lvert \mathrm{Im}\,\lambda\right\rvert\leq
 \dfrac{2K}{2K+1}\,\mathrm{Re}\,\lambda+8K(2K+1)^{3}\right.\right\},\qquad & \lambda_{0} \leq\mathrm{Re}\,\lambda\leq
 \lambda_{1} \vspace{5pt} \\
 \left\{\lambda\in \mathbb{C}\left| \left\lvert \mathrm{Im}\,\lambda\right\rvert\leq
 \dfrac{32}{12^{3/4}}K\,(\mathrm{Re}\,\lambda)^{3/4} \right.\right\},\qquad & \lambda_{1}<\mathrm{Re}\,\lambda,
 \end{cases}
\end{equation*}
with $\lambda_{0}:=-4(2K+1)^{4}$ and $\lambda_{1}:=12(2K+1)^{4}$.
\end{proposition}

Estimates~\eqref{eq_26} result from Proposition~\ref{pr_Prf12} and Remark~\ref{rm_SsqlFm10}.

Now let $\mathrm{Im}\,q\equiv 0$. 
We estimate the lower bound of the operator  $\mathrm{S}(q)$.
From~\eqref{eq_Prf20} for $K\varepsilon\leq 1/2$ we get:
\begin{align}
 (\mathrm{S}(q)u,u)_{L^{2}(\mathbb{R})} & =\lVert u'\rVert_{L^{2}(\mathbb{R})}+t_{Q,\tau}[u]\geq \lVert
u'\rVert_{L^{2}(\mathbb{R})}-K\varepsilon\lVert u'\rVert_{L^{2}(\mathbb{R})}-4K\varepsilon^{-3}\lVert
u\rVert_{L^{2}(\mathbb{R})}^{2} \label{eq_Prf106} \\
& =(1-K\varepsilon)\lVert u'\rVert_{L^{2}(\mathbb{R})}-4K\varepsilon^{-3}\lVert u\rVert_{L^{2}(\mathbb{R})}^{2}
 \geq -4K\varepsilon^{-3}\lVert u\rVert_{L^{2}(\mathbb{R})}^{2}. \notag
\end{align}
The estimates~\eqref{eq_Prf106} with $\varepsilon:=\min\{1, (2K)^{-1}\}$ give us the required result:
\begin{equation*}
 (\mathrm{S}(q)u,u)_{L^{2}(\mathbb{R})}\geq
 \begin{cases}
  -4K\lVert u\rVert_{L^{2}(\mathbb{R})}^{2},\quad & \text{if}\quad K< 1/2, \\
  -32K^{4}\lVert u\rVert_{L^{2}(\mathbb{R})}^{2},\quad & \text{if}\quad K\geq 1/2.
 \end{cases}
\end{equation*}

Thus Theorem~\ref{thMn_C} is proved completely. \hfill{$\square$}

%%%%%%%%%%%%%%%%%%%%%%%%%%%%%%%%%%%%%%%%%%%%%%%%%%%%%%%%%%%%%%%%%%%%%%%%%%%%%%%%%%%%%%%%%%%%%%%%%%%%%%%%%%%%%%%%%%%%%%%%%%%%
\vspace{25pt}

\textit{Acknowledgment}. The authors were partially supported by the grant no. 01/01-12 of National Academy of Science of
Ukraine (under the joint Ukrainian--Russian project of NAS of Ukraine and Russian Foundation of Basic Research).

%%%%%%%%%%%%%%%%%%%%%%%%%%%%%%%%%%%%%%%%%%%%%%%%%%%%%%%%%%%%%%%%%%%%%%%%%%%%%%%%%%%%%%%%%%%%%%%%%%%%%%%%%%%%%%%%%%%%%%%%%%%%
%\newpage
%%%%%%%%%%%%%%%%%%%%%%%%%%%%%%%%%%%%%%%%%%%%%%%%%%%%%%%%%%%%%%%%%%%%%%%%%%%%%%%%%%%%%%%%%%%%%%%%%%%%%%%%%%%%%%%%%%%%%%%%%%%%

\end{document}